\documentclass[12pt,twoside]{amsart}
\usepackage{amsmath, amsthm, amscd, amsfonts, amssymb, graphicx}

\usepackage{enumerate}
\usepackage[colorlinks=true,
linkcolor=blue,
urlcolor=cyan,
citecolor=red]{hyperref}
\usepackage{mathrsfs}
\newtheorem{theorem}{Theorem}[section]

\newtheorem{remark}{Remark}[section]

\newtheorem{corollary}{Corollary}[section]

\numberwithin{equation}{section}

\begin{document}
\title{Some inequalities for interpolational operator means}
\author{Hamid Reza Moradi, Shigeru Furuichi and Mohammad Sababheh}
 \maketitle

\begin{abstract}
Using the properties of geometric mean, we shall show for any $0\le \alpha ,\beta \le 1$,
\[f\left( A{{\nabla }_{\alpha }}B \right)\le f\left( \left( A{{\nabla }_{\alpha }}B \right){{\nabla }_{\beta }}A \right){{\sharp}_{\alpha }}f\left( \left( A{{\nabla }_{\alpha }}B \right){{\nabla }_{\beta }}B \right)\le f\left( A \right){{\sharp}_{\alpha }}f\left( B \right)\]
whenever $f$ is a non-negative operator log-convex function, $A,B\in \mathcal{B}\left( \mathcal{H} \right)$ are positive operators, and $0\le \alpha ,\beta \le 1$. As an application of this operator mean inequality, we present several refinements of the Aujla subadditive inequality for operator monotone decreasing functions.

Also, in a similar way,  we consider some inequalities of Ando's type.  Among other things, it is shown that if  $\Phi $ is a positive linear map, then 
\[\Phi \left( A{{\sharp}_{\alpha }}B \right)\le \Phi \left( \left( A{{\sharp}_{\alpha }}B \right){{\sharp}_{\beta }}A \right){{\sharp}_{\alpha }}\Phi \left( \left( A{{\sharp}_{\alpha }}B \right){{\sharp}_{\beta }}B \right)\le \Phi \left( A \right){{\sharp}_{\alpha }}\Phi \left( B \right).\]
\end{abstract}
\pagestyle{myheadings}
\markboth{\centerline {Some inequalities for interpolational operator means}}
{\centerline {H.R. Moradi, S. Furuichi \& M. Sababheh}}
\bigskip
\bigskip
\section{Introduction and Preliminaries}
We denote the set of all bounded linear operators on a Hilbert space $\mathcal{H}$ by $\mathcal{B}\left( \mathcal{H} \right)$. An operator $A\in \mathcal{B}\left( \mathcal{H} \right)$ is said to be positive (denoted by $A\ge 0$) if $\left\langle Ax,x \right\rangle \ge 0$ for all $x\in \mathcal{H}$. If a positive operator is invertible, it is said to be strictly positive and we write $A>0.$

The axiomatic theory for connections and means for pairs of positive matrices have been studied by Kubo and Ando \cite{kubo}. A binary operation $\sigma$ defined on the cone of strictly positive operators is called an operator mean if for $A,B>0,$
\begin{itemize}
	\item[(i)] $I\sigma I=I$, where $I$ is the identity operator;
	\item[(ii)] ${{C}^{*}}\left( A\sigma B \right)C\le \left( {{C}^{*}}AC \right)\sigma \left( {{C}^{*}}BC \right)$, $\forall C\in\mathcal{B}(\mathcal{H})$;
	\item[(iii)] $A_{n}\downarrow A$ and $B_{n}\downarrow B$ imply $A_{n}\sigma B_{n}\downarrow A\sigma B$, where ${{A}_{n}}\downarrow A$ means that ${{A}_{1}}\ge {{A}_{2}}\ldots $  and ${{A}_{n}}\to A$  as $n\to \infty $ in the strong operator topology;
	\item[(iv)] 
	\begin{equation}\label{13}
A\le B\quad \& \quad C\le D\quad\text{ }\Rightarrow\quad \text{ }A\sigma C\le B\sigma D, \forall C,D>0.	
	\end{equation}
\end{itemize}

For a symmetric operator mean $\sigma $ (in the sense that $A\sigma B=B\sigma A$), a parametrized operator mean ${{\sigma }_{\alpha }}$  ($\alpha \in \left[ 0,1 \right]$) is called an interpolational path for $\sigma $ (or Uhlmann's interpolation for $\sigma $) if it satisfies
\begin{itemize}
	\item[(c1)] 	$A{{\sigma }_{0}}B=A$ (here we recall the convention ${{T}^{0}}=I$ for any positive operator $T$), $A{{\sigma }_{1}}B=B$, and $A{{\sigma }_{\frac{1}{2}}}B=A\sigma B$;
	\item[(c2)] $\left( A{{\sigma }_{\alpha }}B \right)\sigma \left( A{{\sigma }_{\beta }}B \right)=A{{\sigma }_{\frac{\alpha +\beta }{2}}}B$ for all $\alpha ,\beta \in \left[ 0,1 \right]$;
	\item[(c3)] the map $\alpha \in \left[ 0,1 \right]\mapsto A{{\sigma }_{\alpha }}B$  is norm continuous for each $A$ and $B$.
\end{itemize}
It is straightforward to see that  the set of all $\gamma \in \left[ 0,1 \right]$  satisfying
\begin{equation}\label{20}
\left( A{{\sigma }_{\alpha }}B \right){{\sigma }_{\gamma }}\left( A{{\sigma }_{\beta }}B \right)=A{{\sigma }_{\left( 1-\gamma  \right)\alpha +\gamma \beta }}B
\end{equation}
for all $\alpha ,\beta $ is a convex subset of $\left[ 0,1 \right]$  including 0 and 1. Therefore \eqref{20} is valid for all $\alpha ,\beta ,\gamma \in \left[ 0,1 \right]$ (see \cite[Lemma 1]{fujii}).

Typical interpolational means are so-called power means
\[A{{m}_{\upsilon }}B={{A}^{\frac{1}{2}}}{{\left( \frac{1}{2}\left( I+{{\left( {{A}^{-\frac{1}{2}}}B{{A}^{-\frac{1}{2}}} \right)}^{\upsilon }} \right) \right)}^{\frac{1}{\upsilon }}}{{A}^{\frac{1}{2}}},\quad\text{ }-1\le \upsilon \le 1\]
and their interpolational paths are \cite[Theorem 5.24]{mond-pecaric},
\[A{{m}_{\upsilon ,\alpha }}B={{A}^{\frac{1}{2}}}{{\left( \left( 1-\alpha  \right)I+\alpha {{\left( {{A}^{-\frac{1}{2}}}B{{A}^{-\frac{1}{2}}} \right)}^{\upsilon }} \right)}^{\frac{1}{\upsilon }}}{{A}^{\frac{1}{2}}},\quad\text{ }0\le \alpha \le 1.\]
In particular, we have
	\[A{{m}_{1,\alpha }}B=A{{\nabla }_{\alpha }}B=\left( 1-\alpha  \right)A+\alpha B,\] 
\[A{{m}_{0,\alpha }}B=A{{\sharp}_{\alpha }}B={{A}^{\frac{1}{2}}}{{\left( {{A}^{-\frac{1}{2}}}B{{A}^{-\frac{1}{2}}} \right)}^{\alpha }}{{A}^{\frac{1}{2}}},\] 
\[A{{m}_{-1,\alpha }}B=A{{!}_{\alpha }}B={{\left( {{A}^{-1}}{{\nabla }_{\alpha }}B \right)}^{-1}}.\] 
They are called the weighted arithmetic, weighted geometric, and weighted harmonic interpolations respectively. It is well-known that
\begin{equation}\label{15}
A{{!}_{\alpha }}B\le A{{\sharp}_{\alpha }}B\le A{{\nabla }_{\alpha }}B,\quad\text{ }0\le \alpha \le 1
\end{equation}

In \cite{aujla}, Aujla et al. introduced the notion of operator log-convex functions in the following way: A continuous real function $f:\left( 0,\infty  \right)\to \left( 0,\infty  \right)$ is called operator log-convex if
\begin{equation}\label{22}
f\left( A{{\nabla }_{\alpha }}B \right)\le f\left( A \right){{\sharp}_{\alpha }}f\left( B \right),\quad\text{ }0\le \alpha \le 1
\end{equation}
for all positive operators $A$ and $B$. After that, Ando and Hiai \cite{1} gave the following characterization of operator monotone decreasing functions: Let $f$ be a continuous non-negative function on $\left( 0,\infty  \right)$. Then the following conditions are equivalent:
\begin{itemize}
	\item[(a)] $f$ is operator monotone decreasing;
	\item[(b)] $f$ is operator log-convex;
	\item[(c)] $f\left( A\nabla B \right)\le f\left( A \right)\sigma f\left( B \right)$ for all positive operators $A$, $B$
	and for all symmetric operator means $\sigma $.
\end{itemize}

In Theorem \ref{theorem01} below, we provide a more precise estimate than \eqref{22} for operator log-convex functions.  As a by-product, we improve both inequalities in \eqref{15}. Additionally, we give refinement and two reverse inequalities for the triangle inequality.

Our main application of Theorem \ref{theorem01} is a subadditive behavior of operator monotone decreasing functions. Recall that a concave function  (not necessarily operator concave) $f:(0,\infty)\to [0,\infty)$ enjoys the subadditive inequality
\begin{equation}\label{concave_subadditive_intro}
f(a+b)\leq f(a)+f(b), a,b>0.
\end{equation}

Operator concave functions (equivalently, operator monotone) do not enjoy the same subadditive behavior. However, in \cite{ando} 
it was shown that an operator concave function $f:(0,\infty)\to (0,\infty)$ satisfies the norm version of \eqref{concave_subadditive_intro} as follows
\begin{equation}\label{subadditive_ando}
|||f(A+B)|||\leq |||f(A)+f(B)|||,
\end{equation}
for positive definite matrices $A,B$ and any unitraily invariant norm $|||\;\;|||$. Later, the authors in \cite{bourin} showed that \eqref{subadditive_ando} is still valid for concave functions $f:(0,\infty)\to (0,\infty)$ (not necessarily operator concave).

We emphasize that \eqref{subadditive_ando} does not hold without the norm. In \cite{aujla1}, it is shown that an operator monotone decreasing function $f:(0,\infty)\to (0,\infty)$ satisfies the subadditive inequality
\begin{equation}\label{aujla_sub_intro}
f(A+B)\leq f(A)\nabla f(B),
\end{equation}
for the positive matrices $A,B.$

In Corollary \ref{subadditive_coro}, we present multiple refinements of \eqref{aujla_sub_intro}.

The celebrated Ando's inequality  asserts that if $\Phi $ is a positive linear map and $A,B\in \mathcal{B}\left( \mathcal{H} \right)$ are positive operators, then
\begin{equation}\label{23}
\Phi \left( A{{\sharp}_{\alpha }}B \right)\le \Phi \left( A \right){{\sharp}_{\alpha }}\Phi \left( B \right),\quad\text{ }0\le \alpha \le 1.
\end{equation}
Recall that, a linear map $\Phi $ is positive if $\Phi \left( A \right)$ is positive whenever $A$ is positive. We improve and extend this result to Uhlmann's interpolation ${{\sigma }_{\alpha \beta }}$ ($0\le \alpha ,\beta \le 1$). Precisely speaking, we prove that
\[\begin{aligned}
\Phi \left( A{{\sigma }_{\alpha \beta }}B \right)&\le \Phi \left( \left( A{{\sigma }_{\alpha }}B \right){{\sigma }_{\beta }}\left( A{{\sigma }_{0}}B \right) \right){{\sigma }_{\alpha }}\Phi \left( \left( A{{\sigma }_{\alpha }}B \right){{\sigma }_{\beta }}\left( A{{\sigma }_{1}}B \right) \right) \\ 
& \le \Phi \left( A \right){{\sigma }_{\alpha \beta }}\Phi \left( B \right).  
\end{aligned}\]
This result is included in Section \ref{s2}.
\section{On the operator log-convexity}\label{s1}
Our first main result in this paper reads as follows.
\begin{theorem}\label{theorem01}
Let $A,B\in \mathcal{B}\left( \mathcal{H} \right)$ be positive operators and $0\le \alpha \le 1$. If $f$ is a non-negative operator monotone decreasing function, then
\begin{equation}\label{19}
f\left( A{{\nabla }_{\alpha }}B \right)\le f\left( \left( A{{\nabla }_{\alpha }}B \right){{\nabla }_{\beta }}A \right){{\sharp}_{\alpha }}f\left( \left( A{{\nabla }_{\alpha }}B \right){{\nabla }_{\beta }}B \right)\le f\left( A \right){{\sharp}_{\alpha }}f\left( B \right)
\end{equation}
for any $0\le \beta \le 1$.
\end{theorem}
\begin{proof}
Assume $f$ is operator monotone decreasing. We start with the useful identity
\begin{equation}\label{12}
A{{\nabla }_{\alpha }}B=\left( \left( A{{\nabla }_{\alpha }}B \right){{\nabla }_{\beta }} A\right){{\nabla }_{\alpha }}\left( \left( A{{\nabla }_{\alpha }}B \right) {{\nabla }_{\beta }}B\right),
\end{equation}
which follows from \eqref{20} with $A=A\nabla_0B$ and $B=A\nabla_1B$. Then we have
\begin{align}
 f\left( A{{\nabla }_{\alpha }}B \right)&=f\left( \left( \left( A{{\nabla }_{\alpha }}B \right) {{\nabla }_{\beta }}A\right){{\nabla }_{\alpha }}\left( \left( A{{\nabla }_{\alpha }}B \right) {{\nabla }_{\beta }}B\right) \right) \nonumber\\ 
& \le f\left( \left( A{{\nabla }_{\alpha }}B \right) {{\nabla }_{\beta }}A\right){{\sharp}_{\alpha }}f\left( \left( A{{\nabla }_{\alpha }}B \right){{\nabla }_{\beta }}B \right) \label{3}\\ 
& \le \left( f\left( A{{\nabla }_{\alpha }}B \right){{\sharp}_{\beta }}f(A) \right){{\sharp}_{\alpha }}\left( f\left( A{{\nabla }_{\alpha }}B \right){{\sharp}_{\beta }}f(B) \right) \label{4}\\ 
& \le \left(\left( f\left( A \right){{\sharp}_{\alpha }}f\left( B \right) \right) {{\sharp}_{\beta }}f(A)\right){{\sharp}_{\alpha }}\left( \left( f\left( A \right){{\sharp}_{\alpha }}f\left( B \right) \right) {{\sharp}_{\beta }}f(B)\right) \label{5}\\ 
& =\left(\left( f\left( A \right){{\sharp}_{\alpha }}f\left( B \right) \right) {{\sharp}_{\beta}}\left(f(A)\sharp_0f(B)\right)\right){{\sharp}_{\alpha }}\left( \left( f\left( A \right){{\sharp}_{\alpha }}f\left( B \right) \right) {{\sharp}_{\beta }}\left(f(A)\sharp_1f(B)\right)\right)  \label{6}\\ 
& =f\left( A \right){{\sharp}_{(1-\beta)\alpha +\beta\alpha}}f\left( B \right) \label{7}\\ 
& =f\left( A \right){{\sharp}_{\alpha }}f\left( B \right) \nonumber 
\end{align}
where the inequalities \eqref{3}, \eqref{4} and \eqref{5} follow directly from the log-convexity assumption on $f$ together with  \eqref{13}, the equalities \eqref{6} and \eqref{7} are obtained from the property (c1) and \eqref{20}, respectively. 
 This completes the proof.

\end{proof}

As promised in the introduction, we present the following refinement of Aujla inequality \eqref{aujla_sub_intro}, as a main application of Theorem \ref{theorem01}.   

\begin{corollary}\label{subadditive_coro}
Let $A,B\in \mathcal{B}\left( \mathcal{H} \right)$ be positive operators. If $f$ is a non-negative operator monotone decreasing function, then
\begin{align*}
f(A+B)&\leq f(3A\nabla B)\sharp f(A\nabla 3B)\\
&\leq f(2A)\sharp f(2B)\\
&\leq f(2A)\nabla f(2B)\\
&\leq f(A)\nabla f(B).
\end{align*}
\end{corollary}
\begin{proof}
In Theorem \ref{theorem01}, let $\alpha=\beta=\frac{1}{2}$ and replace $(A,B)$ by $(2A,2B).$ This implies the first and second inequalities immediately. The third inequality follows from the second inequality in \eqref{15}, while the last inequality follows properties of operator means and the fact that $f$ is operator monotone decreasing.
\end{proof}

\begin{remark}
Let $A,B\in \mathcal{B}\left( \mathcal{H} \right)$ be positive operators and $0\le \alpha \le 1$. If $f$ is a function satisfying	
\begin{equation}\label{9}
f\left( A{{\nabla }_{\alpha }}B \right)\le f\left( \left( A{{\nabla }_{\alpha }}B \right) {{\nabla }_{\beta }} A\right){{\sharp}_{\alpha }}f\left( \left( A{{\nabla }_{\alpha }}B \right) {{\nabla }_{\beta }}B\right),
\end{equation}
for $0\leq \beta\leq 1,$ then $f$ is operator monotone decreasing. This follows by taking $\beta =1$ in \eqref{9} and equivalence between (a) and (b) above.
\end{remark}

\begin{corollary}
Let $A,B\in \mathcal{B}\left( \mathcal{H} \right)$ be positive operators. If $g$ is a non-negative operator monotone increasing, then
\[g\left( A{{\nabla }_{\alpha }}B \right)\ge g\left( \left( A{{\nabla }_{\alpha }}B \right){{\nabla }_{\beta }}A \right){{\sharp}_{\alpha }}g\left( \left( A{{\nabla }_{\alpha }}B \right){{\nabla }_{\beta }}B \right)\ge g\left( A \right){{\sharp}_{\alpha }}g\left( B \right)\]
for any $0\le \alpha ,\beta \le 1$.
\end{corollary}

\begin{proof}
It was shown in \cite{1} that  operator monotonicity of $g$ is equivalent to  operator log-concavity ( $g\left( A{{\nabla }_{\alpha }}B \right)\ge g\left( A \right){{\sharp}_{\alpha }}g\left( B \right)$). The proof goes in a similar way to the proof of Theorem \ref{theorem01}.
\end{proof}
\begin{remark}
	In \cite[Remark 2.6]{1}, we have for non-negative operator monotone decreasing function $f$, any operator mean $\sigma$ and $A,B>0,$
	\begin{equation}\label{16}
	f(A\nabla_{\alpha}B)\le f(A)!_{\alpha}f(B)\le f(A) \sigma f(B),\; 0\leq \alpha\leq 1.
	\end{equation}
	
Better estimates than \eqref{16} may be obtained as follows, where $0\leq \alpha, \beta\leq 1,$
\begin{align}
 f\left( A{{\nabla }_{\alpha }}B \right)&=f\left( \left( \left( A{{\nabla }_{\alpha }}B \right) {{\nabla }_{\beta }}A\right){{\nabla }_{\alpha }}\left( \left( A{{\nabla }_{\alpha }}B \right) {{\nabla }_{\beta }}B\right) \right) \nonumber\\ 
& \le f\left( \left( A{{\nabla }_{\alpha }}B \right) {{\nabla }_{\beta }}A\right){{!}_{\alpha }}f\left( \left( A{{\nabla }_{\alpha }}B \right){{\nabla }_{\beta }}B \right) \nonumber\\ 
& \le \left( f\left( A{{\nabla }_{\alpha }}B \right){{!}_{\beta }}f(A) \right){{!}_{\alpha }}\left( f\left( A{{\nabla }_{\alpha }}B \right){{!}_{\beta }}f(B) \right) \nonumber\\ 
& \le \left(\left( f\left( A \right){{!}_{\alpha }}f\left( B \right) \right) {{!}_{\beta }}f(A)\right){{!}_{\alpha }}\left( \left( f\left( A \right){{!}_{\alpha }}f\left( B \right) \right) {{!}_{\beta }}f(B)\right) \nonumber\\ 
& =\left(\left( f\left( A \right){{!}_{\alpha }}f\left( B \right) \right) {{!}_{\beta}}\left(f(A)!_0f(B)\right)\right){{!}_{\alpha }}\left( \left( f\left( A \right){{!}_{\alpha }}f\left( B \right) \right) {{!}_{\beta }}\left(f(A)!_1f(B)\right)\right)  \nonumber\\ 
& =f\left( A \right){{!}_{(1-\beta)\alpha +\beta\alpha}}f\left( B \right) \nonumber\\ 
& =f\left( A \right){{!}_{\alpha }}f\left( B \right) \nonumber\\
& \le f\left( A \right){{\sigma}}f\left( B \right) \nonumber
\end{align}
\end{remark}

In the following we improve the well-known weighted operator arithmetic-geometric-harmonic mean inequalities \eqref{15}.
\begin{theorem}\label{amgm_ref}
Let $A,B\in \mathcal{B}\left( \mathcal{H} \right)$ be positive operators. Then
\[\begin{aligned}
 A{{!}_{\alpha }}B&\le \left( \left( A{{\sharp}_{\alpha }}B \right){{\sharp}_{\beta }}A \right){{!}_{\alpha }}\left( \left( A{{\sharp}_{\alpha }}B \right) {{\sharp}_{\beta }}B\right) \\ 
& \le A{{\sharp}_{\alpha }}B \\ 
& \le \left( \left( A{{\sharp}_{\alpha }}B \right){{\sharp}_{\beta }}A \right){{\nabla }_{\alpha }}\left( \left( A{{\sharp}_{\alpha }}B \right) {{\sharp}_{\beta }}B\right) \\ 
& \le A{{\nabla }_{\alpha }}B  
\end{aligned}\]
for  $0\le \alpha ,\beta \le 1$.
\end{theorem}
\begin{proof}
It follows from the proof of Theorem \ref{theorem01} that
\begin{equation}\label{21}
A{{\sharp}_{\alpha }}B=\left( \left( A{{\sharp}_{\alpha }}B \right){{\sharp}_{\beta }}A \right){{\sharp}_{\alpha }}\left( \left( A{{\sharp}_{\alpha }}B \right){{\sharp}_{\beta }}B \right),\quad\text{ }0\le \alpha ,\beta \le 1.
\end{equation}
Thus, we have
\begin{align}
 A{{\sharp}_{\alpha }}B&=\left( \left( A{{\sharp}_{\alpha }}B \right){{\sharp}_{\beta }}A \right){{\sharp}_{\alpha }}\left( \left( A{{\sharp}_{\alpha }}B \right){{\sharp}_{\beta }}B \right) \nonumber\\ 
& \le \left( \left( A{{\sharp}_{\alpha }}B \right) {{\sharp}_{\beta }}A\right){{\nabla }_{\alpha }}\left( \left( A{{\sharp}_{\alpha }}B \right) {{\sharp}_{\beta }}B\right) \label{10}\\ 
& \le \left( \left( A{{\nabla }_{\alpha }}B \right){{\nabla}_{\beta }} A \right) \nabla_{\alpha }\left( \left( A{{\nabla}_{\alpha }}B \right){{\nabla }_{\beta }}B \right) \label{11}\\ 
& =\left( \left( A{{\nabla }_{\alpha }}B \right){{\nabla}_{\beta }} \left(A\nabla_0B\right) \right) \nabla_{\alpha }\left( \left( A{{\nabla}_{\alpha }}B \right){{\nabla }_{\beta }}\left(A\nabla_1B\right) \right) \nonumber \\ 
& =A{{\nabla }_{\alpha }}B \label{11final}  
\end{align}
where in the inequalities \eqref{10} and \eqref{11} we used the weighted operator arithmetic-geometric mean inequality and the equality \eqref{11final} follows from \eqref{20}. This proves the third and fourth inequalities.

As for the first and second inequalities, replace $A$ and $B$ by $A^{-1}$ and $B^{-1}$, respectively, in the third and fourth inequalities 
$$
A{{\sharp}_{\alpha }}B\le \left( \left( A{{\sharp}_{\alpha }}B \right){{\sharp}_{\beta }}A \right){{\nabla }_{\alpha }}\left( \left( A{{\sharp}_{\alpha }}B \right){{\sharp}_{\beta }}B \right)\le A{{\nabla }_{\alpha }}B
$$
which we have just shown. Then take the inverse to obtain the required results (thanks to the identity $A^{-1}\sharp_{\alpha}B^{-1}=(A\sharp_{\alpha}B)^{-1}$). This completes the proof.
\end{proof}

\begin{remark}
We notice that similar inequalities maybe obtained for any symmetric mean $\sigma$, as follows. First, observe that if $\sigma,\tau$ are two symmetric means such that $\sigma\leq \tau$, then the set $T=\{t:0\leq t\leq 1\;{\text{and}}\;\sigma_t\leq \tau_t\}$ is convex. Indeed, assume $t_1,t_2\in T$. Then for the positive operators $A,B$, we have
\begin{align*}
A\sigma_{\frac{t_1+t_2}{2}}B&=(A\sigma_{t_1}B)\sigma (A\sigma_{t_2}B)\\
&\leq (A\tau_{t_1}B)\tau (A\tau_{t_2}B)\\
&=A\tau_{\frac{t_1+t_2}{2}}B,
\end{align*}
where we have used the assumptions $\sigma\leq \tau$ and $t_1,t_2\in T.$ This proves that $T$ is convex, and hence $T=[0,1]$ since $0,1\in T$, trivially. Thus, we have shown that if $\sigma\leq \tau$ then $\sigma_{\alpha}\leq \tau_{\alpha},$ for all $0\leq \alpha\leq 1.$ Now noting that
$$A\sigma_{\alpha} B= \left((A\sigma_{\alpha}B)\sigma_{\beta}A\right)\sigma_{\alpha}\left((A\sigma_{\alpha}B)\sigma_{\beta}B\right),$$ and proceeding as in Theorem \ref{theorem01}, we obtain

\begin{equation}
f\left( A{{\nabla }_{\alpha }}B \right)\le f\left( \left( A{{\nabla }_{\alpha }}B \right){{\nabla }_{\beta }}A \right){{\sigma}_{\alpha }}f\left( \left( A{{\nabla }_{\alpha }}B \right){{\nabla }_{\beta }}B \right)\le f\left( A \right){{\sigma}_{\alpha }}f\left( B \right)
\end{equation}
for any $0\le \beta \le 1$ and the operator log-convex function $f$. This provides a more precise estimate than $(c)$ above.

On the other hand, proceeding as in Theorem \ref{amgm_ref}, we obtain
\begin{equation}
A\sigma_{\alpha} B\leq \left((A\sigma_{\alpha}B)\sigma_{\beta}A\right)\nabla_{\alpha}\left((A\sigma_{\alpha}B)\sigma_{\beta}B\right)\leq A\nabla_{\alpha}B,
\end{equation}
observing that $\sigma_{\alpha}\leq \nabla_{\alpha}.$ This provides a refinement of the latter basic inequality.
\end{remark}

Taking into account \eqref{12}, it follows that
\begin{equation*}
A+B=\alpha A+\left( 1-\alpha  \right)\left( A\nabla B \right)+\alpha B+\left( 1-\alpha  \right)\left( A\nabla B \right).
\end{equation*}
As a consequence of this inequality, we have the following refinement of the well-known triangle inequality
\[\left\| A+B \right\|\le \left\| A \right\|+\left\| B \right\|.\]
\begin{corollary}\label{triangle_ineq01}
Let $A,B\in \mathcal{B}\left( \mathcal{H} \right)$. Then, for $\alpha\in\mathbb{R}$,
\[\left\| A+B \right\|\le \left\| \alpha A+\left( 1-\alpha  \right)\left( A\nabla B \right) \right\|+\left\| \alpha B+\left( 1-\alpha  \right)\left( A\nabla B \right) \right\|\le \left\| A \right\|+\left\| B \right\|.\]
\end{corollary}

\begin{remark}
Using Corollary \ref{triangle_ineq01}, we obtain the reverse triangle inequalities
$$
\left\| A \right\| -\left\| B \right\| \le\frac{1}{2}\left(\left\| A \nabla_{-\alpha}(2B)\right\| +\left\| A \nabla_{\alpha}(2B)\right\|-2\left\| B \right\| \right)\le \left\| A-B \right\|
$$
and
$$
\left\| B \right\| -\left\| A \right\| \le \frac{1}{2}\left(\left\| B \nabla_{-\alpha}(2A)\right\| +\left\| B \nabla_{\alpha}(2A)\right\|-2\left\| A \right\| \right)\le \left\| A-B \right\|,
$$
where $\alpha\in\mathbb{R}.$
\end{remark}
\section{A glimpse at the Ando's inequality}\label{s2}
In this section, we present some versions and improvements of Ando's inequality \eqref{23}.
\begin{theorem}\label{2}
	Let $A,B\in \mathcal{B}\left( \mathcal{H} \right)$ be positive operators and $\Phi $ be a positive linear map. Then for any $0\le \alpha ,\beta \le 1$,
	\begin{equation}\label{17}
\Phi \left( A{{\sharp}_{\alpha }}B \right)\le \Phi \left( \left( A{{\sharp}_{\alpha }}B \right){{\sharp}_{\beta }}A \right){{\sharp}_{\alpha }}\Phi \left( \left( A{{\sharp}_{\alpha }}B \right){{\sharp}_{\beta }}B \right)\le \Phi \left( A \right){{\sharp}_{\alpha }}\Phi \left( B \right).
	\end{equation}
	In particular,
	\begin{equation}\label{18}
	\begin{aligned}
	\sum\limits_{j=1}^{m}{{{A}_{j}}{{\sharp}_{\alpha }}{{B}_{j}}}&\le \left( \sum\limits_{j=1}^{m}{\left( {{A}_{j}}{{\sharp}_{\alpha }}{{B}_{j}} \right){{\sharp}_{\beta }}    {{A}_{j}}    } \right){{\sharp}_{\alpha }}\left( \sum\limits_{j=1}^{m}{ \left( {{A}_{j}}{{\sharp}_{\alpha }}{{B}_{j}} \right){{\sharp}_{\beta }}   {{B}_{j}}      } \right) \\ 
	& \le \left( \sum\limits_{j=1}^{m}{{{A}_{j}}} \right){{\sharp}_{\alpha }}\left( \sum\limits_{j=1}^{m}{{{B}_{j}}} \right).  
	\end{aligned}
	\end{equation}
\end{theorem}
\begin{proof}
	We omit the proof of \eqref{17} because it is proved in a way similar to that of \eqref{19} in Theorem \ref{theorem01}. Now, if in \eqref{17} we take $\Phi :{{M}_{nk}}\left( \mathbb{C} \right)\to {{M}_{k}}\left( \mathbb{C} \right)$ defined by 
	\[\Phi \left( \left( \begin{matrix}
	{{X}_{1,1}} & {} & {}  \\
	{} & \ddots  & {}  \\
	{} & {} & {{X}_{n,n}}  \\
	\end{matrix} \right) \right)={{X}_{1,1}}+\ldots +{{X}_{n,n}}\] 
	and apply $\Phi $ to $A={\text{diag}}\left( {{A}_{1}},\ldots ,{{A}_{n}} \right)$ and $B={\text{diag}}\left( {{B}_{1}},\ldots ,{{B}_{n}} \right)$, we get \eqref{18}. 
\end{proof}

In the following, we present a more general form of \eqref{17} will be shown.

\begin{theorem}
Let $A,B\in \mathcal{B}\left( \mathcal{H} \right)$ be positive operators and $\Phi $ be any positive linear map. Then we have the following inequalities for Uhlmann's interpolation ${{\sigma }_{\alpha \beta }}$ and $0\le \alpha ,\beta \le 1$,
\[\begin{aligned}\label{theorem3.2}
\Phi \left( A{{\sigma }_{\alpha \beta }}B \right)&\le \Phi \left( \left( A{{\sigma }_{\alpha }}B \right){{\sigma }_{\beta }}\left( A{{\sigma }_{0}}B \right) \right){{\sigma }_{\alpha }}\Phi \left( \left( A{{\sigma }_{\alpha }}B \right){{\sigma }_{\beta }}\left( A{{\sigma }_{1}}B \right) \right) \\ 
& \le \Phi \left( A \right){{\sigma }_{\alpha \beta }}\Phi \left( B \right).  
\end{aligned}\]
\end{theorem}
\begin{proof}
Thanks to \eqref{20}, we obviously have 
\[\begin{aligned}
 &\left( \left( A{{\sigma }_{\alpha }}B \right){{\sigma }_{\beta }}\left( A{{\sigma }_{0}}B \right) \right){{\sigma }_{\alpha }}\left( \left( A{{\sigma }_{\alpha }}B \right){{\sigma }_{\beta }}\left( A{{\sigma }_{1}}B \right) \right)\\
& =\left( A{{\sigma }_{\alpha \left( 1-\beta  \right)}}B \right){{\sigma }_{\alpha }}\left( A{{\sigma }_{\alpha \left( 1-\beta  \right)+\beta }}B \right) \\ 
&= A{{\sigma }_{\alpha \beta }}B.  
\end{aligned}\]
Now, the desired result follows directly from the above identities.
\end{proof}

\begin{remark}
From simple calculations, we have the following inequalities for positive operators $A,B\in \mathcal{B}\left( \mathcal{H} \right)$, any positive linear map $\Phi$ and $0\le \alpha ,\beta,\gamma,\delta \le 1$,
\begin{equation}\label{14}
\begin{aligned}
\Phi \left( A{{\sigma }_{\alpha \left( 1-\beta  \right)+\beta \left( \left( 1-\alpha  \right)\gamma +\alpha \delta  \right)}}B \right)&\le \Phi \left( \left( A{{\sigma }_{\alpha }}B \right){{\sigma }_{\beta }}\left( \left( A{{\sigma }_{\gamma }}B \right) \right) \right){{\sigma }_{\alpha }}\Phi \left( \left( A{{\sigma }_{\alpha }}B \right){{\sigma }_{\beta }}\left( \left( A{{\sigma }_{\delta }}B \right) \right) \right) \\ 
& \le \Phi \left( A \right){{\sigma }_{\alpha \left( 1-\beta  \right)+\beta \left( \left( 1-\alpha  \right)\gamma +\alpha \delta  \right)}}\Phi \left( B \right).  
\end{aligned}
\end{equation}
Apparently, \eqref{14} reduces to \eqref{theorem3.2} when $\gamma =0$ and $\delta =1$.
\end{remark}

\vskip 0.5 true cm 

{\tiny (H. R. Moradi) Young Researchers and Elite Club, Mashhad Branch, Islamic Azad University, Mashhad, Iran.
	
	\textit{E-mail address:} hrmoradi@mshdiau.ac.ir}

\vskip 0.3 true cm 	

{\tiny (S. Furuichi) Department of Information Science, College of Humanities and Sciences, Nihon University, 3-25-40, Sakurajyousui,	Setagaya-ku, Tokyo, 156-8550, Japan. }

{\tiny \textit{E-mail address:} furuichi@chs.nihon-u.ac.jp }

\vskip 0.3 true cm

{\tiny (M. Sababheh) Department of Basic Sciences, Princess Sumaya University for Technology, Amman 11941,
	Jordan. 
	
\textit{E-mail address:} sababheh@yahoo.com; sababheh@psut.edu.jo}

\end{document}